\documentclass{amsart}
\usepackage{eurosym}
\usepackage{amssymb}
\usepackage{amsfonts}

\newtheorem{theorem}{Theorem}[section]
\newtheorem{corollary}[theorem]{Corollary}
\theoremstyle{definition}
\newtheorem{definition}[theorem]{Definition}

\theoremstyle{remark}
\newtheorem{remark}[theorem]{Remark}
\numberwithin{equation}{section}

\input{tcilatex}

\begin{document}
\title[Non-Newtonian metric and multiplicative metric equivalent metric]{A
note on the equivalence of some metric, non-Newtonian and multiplicative
metric results }
\author{Bahri Turan }
\author{C\"{u}neyt \c{C}evik}
\address{Department of Mathematics, Faculty of Science, Gazi University \\
06500 Teknikokullar Ankara, Turkey}
\email{bturan@gazi.edu.tr\\
ccevik@gazi.edu.tr }
\date{March 2, 2016}

\begin{abstract}
In this short note is on the equivalence between non-Newtonian metric
(particularly multiplicative metric) and metric. We present a different
proof the fact that the notion of a non-Newtonian metric space is not more
general than that of a metric space. Also, we emphasize that a lot of fixed
point results in the non-Newtonian metric setting can be directly obtained
from their metric counterparts.
\end{abstract}

\maketitle

\section{Introduction and preliminaries}

Recently, many papers dealing with non-Newtonian metric and the
multiplicative metric have been published [2-10]. Although the
multiplicative metric and non-Newtonian metric were announced as a new
theory for topology and fixed point theory, many studies can be obtained
with a simple observation. In the present work we show that some topological
results of non-Newtonian metric can be obtained in an easier way. Therefore,
a lot of fixed point and common fixed point results from the metric setting
can be proved in the non-Newtonian metric (particularly the multiplicative
metric) setting.

Arithmetic is any system that satisfies the whole of the ordered field
axioms whose domain is a subset of $%
\mathbb{R}
$. There are infinitely many types of arithmetic, all of which are
isomorphic, that is, structurally equivalent. In non-Newtonian calculus, a 
\textit{generator} $\alpha $ is a one-to-one function whose domain is all
real numbers and whose range is a subset of real numbers. Each generator
generates exactly one arithmetic, and conversely each arithmetic is
generated by exactly one generator. By \textit{$\alpha $-arithmetic}, we
mean the arithmetic whose operations and whose order are defined as 
\begin{equation*}
\begin{array}{lrcll}
\alpha \text{-\textit{addition}} & x\ \dot{+}\ y & = & \alpha \{\alpha
^{-1}(x)+\alpha ^{-1}(y)\} &  \\ 
\alpha \text{-\textit{subtraction}} & x\ \dot{-}\ y & = & \alpha \{\alpha
^{-1}(x)-\alpha ^{-1}(y)\} &  \\ 
\alpha \text{-\textit{multiplication}} & x\ \dot{\times}\ y & = & \alpha
\{\alpha ^{-1}(x)\times \alpha ^{-1}(y)\} &  \\ 
\alpha \text{-\textit{division}} & x\ \dot{/}\ y & = & \alpha \{\alpha
^{-1}(x)\div \alpha ^{-1}(y)\} & (\alpha ^{-1}(y)\neq 0) \\ 
\alpha \text{-\textit{order}} & x\ \dot{<}\ y & \Leftrightarrow & \alpha
^{-1}(x)<\alpha ^{-1}(y) & 
\end{array}%
\end{equation*}%
for all $x$ and $y$ in the range $%
\mathbb{R}
_{\alpha }$ of $\alpha $. In the special cases the identity function $I$ and
the exponential function $\exp $ generate the classical and geometric
arithmetics, respectively.\smallskip

\begin{equation*}
\begin{tabular}{ccccccc}
\hline
$\alpha $ &  & $\alpha \text{-\textit{addition}}$ & $\alpha \text{-\textit{%
subtraction}}$ & $\alpha \text{-\textit{multiplication}}$ & $\alpha \text{-%
\textit{division}}$ & $\alpha \text{-\textit{order}}$ \\ \hline
$I$ &  & $x+y$ & $x-y$ & $xy$ & $x/y$ & $x<y\medskip $ \\ 
$\exp $ &  & $xy$ & $x/y$ & $x^{\ln y}$ $\left( y^{\ln x}\right) $ & $%
x^{1/\ln y}$ & $\ln x<\ln y$ \\ \hline
\end{tabular}%
\end{equation*}%
\smallskip

For further information about $\alpha $-arithmetics, we refer to [1].

Now, we give the definitions of non-Newtonian metric [2] and multiplicative
metric [3] with new notations.

\begin{definition}
Let $X$ be a non-empty set and let $%
\mathbb{R}
_{\alpha }$ be an ordered field generated by a generator $\alpha $ on $%
\mathbb{R}
$. The map $d^{\alpha }:X\times X\rightarrow 
\mathbb{R}
_{\alpha }$ is said to be a \textit{non-Newtonian metric} if it satisfies
the following properties:

($\alpha $m1) $\dot{0}=\alpha (0)\ \dot{\leq}\ d^{\alpha }(x,y)$ and $%
d^{\alpha }(x,y)=\dot{0}$ $\Leftrightarrow $ $x=y,$

($\alpha $m2) $d^{\alpha }(x,y)=d^{\alpha }(y,x)$

($\alpha $m3) $d^{\alpha }(x,y)\ \dot{\leq}\ d^{\alpha }(x,z)\ \dot{+}\
d^{\alpha }(z,y)$ \newline
for all $x,y,z\in X$. Also the pair $(X,d^{\alpha })$ is said to be a 
\textit{non-Newtonian metric space}.
\end{definition}

Thus we can define a new distance for every generator $\alpha .$ When $%
\alpha =\exp $, the non-Newtonian metric $d^{\exp }$ is called
multiplicative metric. Then, $%
\mathbb{R}
_{\exp }=%
\mathbb{R}
_{+}$ and $\dot{0}=1$.

\begin{definition}
Let $X$ be a non-empty set. The map $d^{\exp }:X\times X\rightarrow 
\mathbb{R}
_{+}$ is said to be a \textit{multiplicative metric} if it satisfies the
following properties:

(mm1) $1\leq \ d^{\exp }(x,y)$ and $d^{\exp }(x,y)=1$ $\Leftrightarrow $ $%
x=y,$

(mm2) $d^{\exp }(x,y)=d^{\exp }(y,x)$

(mm3) $d^{\exp }(x,y)\leq d^{\exp }(x,z).d^{\exp }(z,y)$ \newline
for all $x,y,z\in X$. Also the pair $(X,d^{\exp })$ is said to be a \textit{%
multiplicative metric space}.
\end{definition}

\section{Main results}

Let $\alpha $ be a generator on $%
\mathbb{R}
$ and $%
\mathbb{R}
_{\alpha }=\{\alpha (u):u\in 
\mathbb{R}
\}$. By the injectivity of $\alpha $ we have%
\begin{equation*}
\begin{array}{rclc}
\alpha (u+v) & = & \alpha (u)\ \dot{+}\ \alpha (v) &  \\ 
\alpha (u-v) & = & \alpha (u)\ \dot{-}\ \alpha (v) &  \\ 
\alpha (u\times v) & = & \alpha (u)\ \dot{\times}\ \alpha (v) &  \\ 
\alpha (u\ /\ v) & = & \alpha (u)\ \dot{/}\ \alpha (v) & (v\neq 0) \\ 
u\leq v & \Leftrightarrow & \alpha (u)\ \dot{\leq}\ \alpha (v) & 
\end{array}%
\ \ \text{and }\ 
\begin{array}{rcl}
\alpha ^{-1}(x\ \dot{+}\ y) & = & \alpha ^{-1}(x)+\alpha ^{-1}(y) \\ 
\alpha ^{-1}(x\ \dot{-}\ y) & = & \alpha ^{-1}(x)-\alpha ^{-1}(y) \\ 
\alpha ^{-1}(x\ \dot{\times}\ y) & = & \alpha ^{-1}(x)\times \alpha ^{-1}(y)
\\ 
\alpha ^{-1}(x\ \dot{/}\ y) & = & \alpha ^{-1}(x)\ /\ \alpha ^{-1}(y) \\ 
x\ \dot{\leq}\ y & \Leftrightarrow & \alpha ^{-1}(x)\leq \alpha ^{-1}(y)%
\end{array}%
\end{equation*}%
for all $x,y\in 
\mathbb{R}
_{\alpha }$ with $u,v\in 
\mathbb{R}
$, $x=\alpha (u)$, $y=\alpha (v)$. Therefore, $\alpha $ and $\alpha ^{-1}$
preserve basic operations and order.

\begin{remark}
\label{ref}Since the generator $\alpha $ and $\alpha ^{-1}$ are order
preserving, for any two elements $x$ and $y$ in $%
\mathbb{R}
_{\alpha }$, $x\ \dot{\leq}\ y$ if and only if $x\leq y$.
\end{remark}

Let $(X,d^{\alpha })$ be a non-Newtonian metric space. For any $\varepsilon
\ \dot{>}\ \dot{0}$ and any $x\in X$ the set 
\begin{equation*}
B_{\alpha }(x,\varepsilon )=\{y\in X:d^{\alpha }(x,y)\ \dot{<}\ \varepsilon
\}
\end{equation*}%
is called an $\alpha $\textit{-open ball of center }$x$\textit{\ and radius }%
$\varepsilon $. A topology on $X$ is obtained easily by defining open sets
as in the classical metric spaces.

Now, let us emphasize that former topology is obtained by real-valued metric
and vice versa.

\begin{theorem}
\label{thm}For any generator $\alpha $ on $%
\mathbb{R}
$ and for any non-empty set $X$ \newline
(1) If $d^{\alpha }$ is a non-Newtonian metric on $X$, then $d=\alpha
^{-1}\circ d^{\alpha }$ is a metric on $X$,\newline
(2) If $d$ is a metric on $X$, then $d^{\alpha }=\alpha \circ d$ is a
non-Newtonian metric on $X$.
\end{theorem}

\begin{proof}
It is obvious that $\alpha $ and $\alpha ^{-1}$ are one-to-one and order
preserving.
\end{proof}

\begin{corollary}
\label{cor1}For any generator $\alpha $ on $%
\mathbb{R}
$ and, let $d^{\alpha }$ and $d$ be a non-Newtonian metric and a metric on a
non-empty set $X$, respectively, as in Theorem \ref{thm}. If $\tau _{\alpha
} $ and $\tau $ are metric topologies induced by $d^{\alpha }$ and $d$,
respectively, then $\tau _{\alpha }=\tau $.
\end{corollary}

\begin{proof}
Since $\delta _{\varepsilon }=\alpha ^{-1}(\varepsilon )>0$ and $\varepsilon
_{\delta }=\alpha (\delta )\ \dot{>}\ \dot{0}$ for all $\varepsilon \ \dot{>}%
\ \dot{0},\delta >0$, we have%
\begin{eqnarray*}
B_{\alpha }(x,\varepsilon _{\delta }) &=&\{y\in X:d^{\alpha }(x,y)\ \dot{<}\
\varepsilon _{\delta }\}=\{y\in X:\alpha \left( d(x,y)\right) \ \dot{<}\
\alpha (\delta )\} \\
&=&\{y\in X:d(x,y)<\delta _{\varepsilon }\}=B(x,\delta _{\varepsilon })
\end{eqnarray*}%
for all $x\in X,\varepsilon \ \dot{>}\ \dot{0},\delta >0$. Therefore, $\tau
_{\alpha }=\tau $.
\end{proof}

\begin{corollary}
\label{cor2}Under the hypothesis of Corollary \ref{cor1}, the topological
properties of $(X,d)$ and $(X,d^{\alpha })$ are equivalent. In particular,
for a sequence $(x_{n})$ in $X$ and for an element $x\in X$\newline
(1) $x_{n}\overset{d^{\alpha }}{\rightarrow }x$ if and only $x_{n}\overset{d}%
{\rightarrow }x$,\newline
(2) $(x_{n})$ is $d^{\alpha }$-Cauchy if and only if $(x_{n})$ is $d$%
-Cauchy, and\newline
(3) $(X,d^{\alpha })$ is complete if and only if $(X,d)$ is complete.
\end{corollary}

\section{Conclusion}

The topological results obtained by non-Newtonian metrics (particularly
multiplicative metrics) as in [2-10] are equivalent the ones obtained by
metrics. In [3-9] some results of the multiplicative metric and in [10] some
results of the non-Newtonian metric have been obtained for the fixed point
theory. Those results are direct consequences of Theorem \ref{thm} and
Corollary \ref{cor2} since any type of contraction mapping for the
non-Newtonian metric space is also a contraction mapping for the metric
space and vice versa. For example, the non-Newtonian contraction $%
T:X\rightarrow X$ as in [10] is the classical Banach contraction since%
\begin{equation}
d^{\alpha }(T(x),T(y))\ \dot{\leq}\ k\dot{\times}d^{\alpha }(x,y)\
\Leftrightarrow \ d(T(x),T(y))\leq \lambda .d(x,y)  \label{1}
\end{equation}%
for all $x,y\in X$ where $k\in \lbrack \alpha (0),\alpha (1))$ is constant, $%
d=\alpha ^{-1}\circ d^{\alpha }$ and $\lambda =\alpha ^{-1}(k)$. In
particular, by Remark \ref{ref} and by (\ref{1}), the multiplicative
contraction $T:X\rightarrow X$ as in [2] is the classical Banach contraction
since%
\begin{equation*}
\begin{array}{rcl}
d^{\exp }(T(x),T(y))\leq d^{\exp }(x,y)^{\lambda } & \!\!\Leftrightarrow \!\!
& d^{\exp }(T(x),T(y))\ \dot{\leq}\ d^{\exp }(x,y)^{\lambda }=k\dot{\times}%
d^{\exp }(x,y) \\ 
& \!\!\Leftrightarrow \!\! & d(T(x),T(y))\leq \lambda .d(x,y)%
\end{array}%
\end{equation*}%
for all $x,y\in X$ where $\lambda \in \lbrack 0,1)$ is constant, $d=\ln
\circ d^{\exp }$ and $\lambda =\ln k$. In this way we can obtain most of the
non-Newtonian metric results and most of the multiplicative metric results
applying corresponding properties from the metric setting.


\begin{thebibliography}{99}
\bibitem{} M. Grossman, R. Katz, Non-Newtonian Calculus, Lee Press, Pigeon
Cove, MA, 1972.

\bibitem{} A.F. \c{C}akmak, F. Ba\c{s}ar, Some new results on sequence
spaces with respect to non-Newtonian calculus, J. Inequal. Appl. 2012,
doi:10.1186/1029-242X-2012-228, 17 pp.

\bibitem{} M. \"{O}zav\c{s}ar, A.C. \c{C}evikel, Fixed points of
multiplicative contraction mappings on multiplicative metric spaces,
arXiv:1205.5131v1, 2012.

\bibitem{} F. Gu, Y.J. Cho, Common fixed point results for four maps
satisfying $\phi $-contractive condition in multiplicative metric spaces,
Fixed Point Theory Appl. 2015, 2015:165, 19 pp.

\bibitem{} M. Abbas, M.De la Sen, T. Nazir, Common fixed points of
generalized rational type cocyclic mappings in multiplicative metric spaces,
Discrete Dyn. Nat. Soc. 2015, Art. ID 532725, 10 pp.

\bibitem{} C. Mongkolkeha, W. Sintunavarat, Best proximity points for
multiplicative proximal contraction mapping on multiplicative metric spaces,
J. Nonlinear Sci. Appl. 8 (2015), no. 6, 1134-1140.

\bibitem{} M. Abbas, B. Ali, Y.I. Suleiman, Common fixed points of locally
contractive mappings in multiplicative metric spaces with application, Int.
J. Math. Math. Sci. 2015, Art. ID 218683, 7 pp.

\bibitem{} O. Yamaod, W. Sintunavarat, Some fixed point results for
generalized contraction mappings with cyclic $(\alpha ,\beta )$-admissible
mapping in multiplicative metric spaces, J. Inequal. Appl. 2014, 2014:488,
15 pp.

\bibitem{} X. He, M. Song, D. Chen, Common fixed points for weak commutative
mappings on a multiplicative metric space, Fixed Point Theory Appl. 2014,
2014:48, 9 pp.

\bibitem{} D. Binba\c{s}\i o\u{g}lu, S. Demiriz, D. T\"{u}rko\u{g}lu, Fixed
points of non-Newtonian contraction mappings on non-Newtonian metric spaces,
J. Fixed Point Theory Appl. 2015, doi: 10.1007/s11784-015-0271-y, 12 pp.
\end{thebibliography}
\end{document}